\documentclass[a4paper,11pt]{amsart}

\usepackage[utf8]{inputenc}
\usepackage{amsmath}
\usepackage{amsfonts}
\usepackage{amssymb}
\usepackage{amsthm}
\usepackage{enumerate}
\usepackage[margin=1in]{geometry}
\usepackage{comment}
\usepackage{color}
\usepackage{hyperref}
\usepackage[british]{babel}

\newcommand{\supp}{\operatorname{supp}}
\newcommand{\link}{\operatorname{link}}
\newcommand{\Z}{\mathbb{Z}}
\newcommand{\R}{\mathbb{R}}
\newcommand{\C}{\mathbb{C}}
\newcommand{\dc}{\operatorname{dc}}
\newcommand{\bb}{\mathfrak{B}}
\newcommand{\sss}{\mathfrak{S}}
\newcommand{\hh}{\mathbf{H}}
\newcommand{\hhh}{\mathbf{\hat{H}}}
\renewcommand{\Re}{\operatorname{Re}}
\renewcommand{\Im}{\operatorname{Im}}
\newtheorem{thm}{Theorem}
\newtheorem{prop}[thm]{Proposition}
\newtheorem{lem}[thm]{Lemma}

\newtheorem{ex}[thm]{Example}

\newtheorem*{ack}{Acknowledgements}
\theoremstyle{definition}
\newtheorem{defn}[thm]{Definition}
\newtheorem{rmk}[thm]{Remark}

\newtheoremstyle{TheoremNum}{\parskip}{\parskip}{\itshape}{}{\bfseries}{$\mathbf{'}$.}{ }{\thmname{#1}\thmnote{ \bfseries #3}}
\theoremstyle{TheoremNum}

\begin{document}

\title[Rational growth and degree of commutativity of graph products]{Rational growth and degree of commutativity \\ of graph products}
\author{Motiejus Valiunas}
\address{Mathematical Sciences \\ University of Southampton \\ University Road \\ Southampton SO17 1BJ \\ United Kingdom}
%\ead{m.valiunas@soton.ac.uk}
\email{m.valiunas@soton.ac.uk}
%\date{\today}
%\begin{keyword}
%graph products of groups \sep degree of commutativity \sep rational growth series
%\MSC[2010]{20F36, 20P05}
%\end{keyword}

\begin{abstract}
Let $G$ be an infinite group and let $X$ be a finite generating set for $G$ such that the growth series of $G$ with respect to $X$ is a rational function; in this case $G$ is said to have rational growth with respect to $X$. In this paper a result on sizes of spheres (or balls) in the Cayley graph $\Gamma(G,X)$ is obtained: namely, the size of the sphere of radius $n$ is bounded above and below by positive constant multiples of $n^\alpha \lambda^n$ for some integer $\alpha \geq 0$ and some $\lambda \geq 1$.

As an application of this result, a calculation of degree of commutativity (d.~c.) is provided: for a finite group $F$, its d.~c.\ is defined as the probability that two randomly chosen elements in $F$ commute, and Antol\'in, Martino and Ventura have recently generalised this concept to all finitely generated groups. It has been conjectured that the d.~c.\ of a group $G$ of exponential growth is zero. This paper verifies the conjecture (for certain generating sets) when $G$ is a right-angled Artin group or, more generally, a graph product of groups of rational growth in which centralisers of non-trivial elements are ``uniformly small''.
\end{abstract}
\maketitle

\section{Introduction}
\label{s:intro}

Let $G$ be a group which has a finite generating set $X$. For any element $g \in G$, let $|g| = |g|_X$ be the word length of $g$ with respect to $X$. For any $n \in \Z_{\geq 0}$, let
\begin{equation*}
B_{G,X}(n) := \{ g \in G \mid |g|_X \leq n \}
\end{equation*}
be the \emph{ball} in $G$ with respect to $X$ of radius $n$, and let
\begin{equation*}
S_{G,X}(n) := \{ g \in G \mid |g|_X = n \}
\end{equation*}
be the \emph{sphere} in $G$ with respect to $X$ of radius $n$. One writes $B_G(n)$ or $B(n)$ for the ball (and $S_G(n)$ or $S(n)$ for the sphere) if the generating set or the group itself is clear. A group $G$ is said to have \emph{exponential growth} if
\begin{equation} \label{e:liminfballs}
\liminf_{n \to \infty} \frac{\log |B_{G,X}(n)|}{n} > 0
\end{equation}
and \emph{subexponential growth} otherwise; note that as there are at most $(2|X|)^n$ words over $X^{\pm 1}$ of length $n$, the limit in \eqref{e:liminfballs} is finite, so the group cannot have `superexponential' growth. A group $G$ is said to have \emph{polynomial growth of degree $d$} if
\begin{equation*}
d := \limsup_{n \to \infty} \frac{\log |B_{G,X}(n)|}{\log n} < \infty
\end{equation*}
and \emph{superpolynomial growth} otherwise. It is well-known that having exponential growth or polynomial growth of degree $d$ is independent of the generating set $X$.

The pairs $(G,X)$ as above considered in this paper will have some special properties. In particular, consider the (\emph{spherical}) \emph{growth series} $s_{G,X}(t)$ of a finitely generated group $G$ with a finite generating set $X$, defined by
\begin{equation*}
s_{G,X}(t) = \sum_{g \in G} t^{|g|_X} = \sum_{n=0}^\infty |S_{G,X}(n)| t^n.
\end{equation*}
Cases of particular interest includes pairs $(G,X)$ for which $s_{G,X}(t)$ is a rational function, i.e.\ a ratio two polynomials; in this case $G$ is said to have \emph{rational growth} with respect to $X$. %This is true when $G$ is a finitely generated virtually abelian group \cite{benson}, a hyperbolic group \cite{gromov87} or the integral Heisenberg group \cite{duchin}, and $X$ is any finite generating set; it is also true for several other classes of groups with respect to certain generating sets.
In general, this property depends on the chosen generating set: for instance, the higher Heisenberg group $G = H_2(\Z)$ has two finite generating sets $X_1$, $X_2$ such that $s_{G,X_1}(t)$ is rational but $s_{G,X_2}(t)$ is not \cite{stoll96}.

Rational growth series implies some nice properties on the growth of a group. In particular, one can obtain the first main result of this paper:
\begin{thm}
\label{t:growth}
Let $G$ be an infinite group with a finite generating set $X$ such that $s_{G,X}(t)$ is a rational function. Then there exist constants $\alpha \in \Z_{\geq 0}$, $\lambda \in [1,\infty)$ and $D > C > 0$ such that 
\begin{equation*}
Cn^\alpha \lambda^n \leq |S_{G,X}(n)| \leq Dn^\alpha \lambda^n
\end{equation*}
for all $n \geq 1$.
\end{thm}

Some of the ideas that go into the proof of Theorem \ref{t:growth} appear in the work of Stoll \cite{stoll96}, where asymptotics of ball sizes are used to show that the higher Heisenberg group $G = H_2(\Z)$ has a finite generating set $X$ such that the series $s_{G,X}(t)$ is transcendental.

\begin{rmk}
It is clear that, with the assumptions and notation as above, Theorem \ref{t:growth} implies
\[
\liminf_{n \to \infty} \frac{|S_{G,X}(n)|}{n^\alpha \lambda^n} \geq C > 0 \qquad \text{and} \qquad \limsup_{n \to \infty} \frac{|S_{G,X}(n)|}{n^\alpha \lambda^n} \leq D < \infty.
\]
It is easy to check that the converse implication is also true. In particular, the conclusion of Theorem \ref{t:growth} is equivalent to the statement that there exist $\alpha \in \Z_{\geq 0}$ and $\lambda \in [1,\infty)$ such that
\[
\liminf_{n \to \infty} \frac{|S_{G,X}(n)|}{n^\alpha \lambda^n} > 0 \qquad \text{and} \qquad \limsup_{n \to \infty} \frac{|S_{G,X}(n)|}{n^\alpha \lambda^n} < \infty.
\]
\end{rmk}

Theorem \ref{t:growth} agrees with the result for hyperbolic groups. Indeed, it is known that if $G$ is a hyperbolic group and $X$ is a finite generating set, then $s_{G,X}(t)$ is rational \cite[Theorem 8.5.N]{gromov87}. In this case the Theorem gives a weaker version of \cite[Th\'eor\`eme 7.2]{coornaert}, which states that the conclusion of Theorem \ref{t:growth} holds with $\alpha = 0$.

As an application of Theorem \ref{t:growth} a calculation of degree of commutativity is provided. For a finite group $F$, the \emph{degree of commutativity} of $F$ was defined by Erd\H{o}s and Tur\'an \cite{erdos} and Gustafson \cite{gustafson} as
\begin{equation}
\label{e:defdcfin}
\dc(F) := \frac{|\{(x,y) \in F^2 \mid [x,y] = 1\}|}{|F|^2},
\end{equation}
i.e.\ the probability that two elements of $F$ chosen uniformly at random commute. In \cite{amv}, Antol\'{i}n, Martino and Ventura generalise this definition to infinite finitely generated groups:
\begin{defn}
Let $G$ be a finitely generated group and let $X$ be a finite generating set for $G$. The \emph{degree of commutativity} for $G$ with respect to $X$ is
\begin{equation*}
\begin{aligned}
\dc_X(G) := &\limsup_{n \to \infty} \frac{|\{(x,y) \in B_{G,X}(n)^2 \mid [x,y] = 1\}|}{|B_{G,X}(n)|^2} \\ = &\limsup_{n \to \infty} \frac{\sum_{x \in B_{G,X}(n)} |C_G(x) \cap B_{G,X}(n)|}{|B_{G,X}(n)|^2},
\end{aligned}
\end{equation*}
where $C_G(x)$ is the centraliser of $x$ in $G$.
\end{defn}
Note that if $G$ is finite then for any generating set $X$ one has $B_{G,X}(N) = G$ for all sufficiently large $N$, so this definition agrees with \eqref{e:defdcfin}.

It is known that $\dc_X(G) = 0$ when $G$ is either a non-virtually-abelian residually finite group of subexponential growth \cite[Theorem 1.3]{amv} or a non-elementary hyperbolic group \cite[Theorem 1.7]{amv}, independently of the generating set $X$. It has been conjectured that indeed $\dc_X(G) = 0$ whenever $G$ has superpolynomial growth \cite[Conjecture 1.6]{amv}.

The interest of this paper is the degree of commutativity of graph products of groups.
\begin{defn}
Let $\Gamma$ be a finite simple (undirected) graph, and let $\hh: V(\Gamma) \to \mathcal{G}$ be a map from the vertex set of $\Gamma$ to the category $\mathcal{G}$ of groups; suppose that $\hh(v) \ncong \{1\}$ for each $v \in V(\Gamma)$. Let
\begin{equation*}
\tilde{G}(\Gamma,\hh) := \ast_{v \in V(\Gamma)} \hh(v)
\end{equation*}
be a free product of groups, and let
\begin{equation*}
R(\Gamma,\hh) := \{ [g,h] \mid g \in \hh(v), h \in \hh(w), \{v,w\} \in E(\Gamma) \}.
\end{equation*}
Then the \emph{graph product} associated with $\Gamma$ and $\hh$ is defined to be the group
\begin{equation*}
G(\Gamma,\hh) := \tilde{G}(\Gamma,\hh) / \langle\!\langle R(\Gamma,\hh) \rangle\!\rangle^{\tilde{G}(\Gamma,\hh)}.
\end{equation*}
\end{defn}
In particular, this is the construction of \emph{right-angled Artin} (respectively \emph{Coxeter}) \emph{groups} if $\hh(v) \cong \Z$ (respectively $\hh(v) \cong C_2$) for all $v \in \Gamma$.

This paper considers groups $G$ which, together with their finite generating sets $X$, belong to a certain class, defined as follows.
\begin{defn}
\label{d:rp}
Say a pair $(G,X)$ with a group $G$ and a finite generating set $X$ of $G$ is a \emph{rational pair with small centralisers} if the following two conditions hold:
\begin{enumerate}[(i)]
\item \label{i:drp-rat} $s_{G,X}(t)$ is a rational;
\item \label{i:drp-sc} there exist constants $P,\beta \in \Z_{\geq 1}$ such that $|C_G(g) \cap B_{G,X}(n)| \leq P n^\beta$ for all $n \geq 1$ and all non-trivial elements $g \in G$.
\end{enumerate}
\end{defn}

Note that condition \eqref{i:drp-sc} is independent of the choice of a generating set $X$: indeed, as any word metrics on $G$ associated with generating sets $X$ and $\hat{X}$ are bi-Lipschitz equivalent, the inequality $|C_G(g) \cap B_{G,X}(n)| \leq P n^\beta$ implies the inequality $|C_G(g) \cap B_{G,\hat{X}}(n)| \leq \hat{P} n^\beta$ for some $\hat{P} \in \Z_{\geq 1}$ depending only on $\hat{X}$ and $P$.

It was shown in \cite{chiswell} that, given a finite simple graph $\Gamma$ with a group $\hh(v)$ and a finite generating set $X(v) \subseteq \hh(v)$ associated to every vertex $v \in V(\Gamma)$, if $s_{\hh(v),X(v)}(t)$ is rational for each $v \in V(\Gamma)$ then so is $s_{G(\Gamma,\hh),X(\Gamma,\hh)}(t)$, where $X(\Gamma,\hh) = \bigsqcup_{v \in V(\Gamma)} X(v)$.

If $G(\Gamma,\hh)$ has exponential growth, then, together with an explicit form of centralisers in $G(\Gamma,\hh)$, described in \cite{barkauskas}, Theorem \ref{t:growth} can be used to compute the degree of commutativity of $G(\Gamma,\hh)$:
\begin{thm}
\label{t:dc}
Let $\Gamma$ be a finite simple graph, and for each vertex $v \in V(\Gamma)$, let $(\hh(v),X(v))$ be a rational pair with small centralisers. Suppose that $G(\Gamma,\hh)$ has exponential growth, and let $X = \bigsqcup_{v \in V(\Gamma)} X(v)$. Then
\begin{equation*}
\dc_{X}(G(\Gamma,\hh)) = 0.
\end{equation*}
\end{thm}

\begin{rmk}
Theorem \ref{t:dc} is enough to confirm \cite[Conjecture 1.6]{amv} in this setting: that is, either $G = G(\Gamma,\hh)$ is virtually abelian, or $\dc_X(G) = 0$. Indeed, $G(\Gamma,\hh)$ has subexponential growth if and only if all the $\hh(v)$ have subexponential growth, the complement $\Gamma^C$ of $\Gamma$ contains no length $2$ paths, and $\hh(v) \cong C_2$ for every non-isolated vertex $v$ of $\Gamma^C$. In this case, rationality of $s_{\hh(v),X(v)}(t)$ implies that the $\hh(v)$ all have polynomial growth (by Theorem \ref{t:growth}, for instance). Thus $G(\Gamma,\hh)$ is a direct product of groups of polynomial growth: namely, the group $\hh(v)$ for each isolated vertex $v$ of $\Gamma^C$, and an infinite dihedral group for each edge in $\Gamma^C$. Consequently, $G(\Gamma,\hh)$ itself has polynomial growth, and so \cite[Corollary 1.5]{amv} implies that either $G(\Gamma,\hh)$ is virtually abelian, or $\dc_X(G(\Gamma,\hh)) = 0$.
\end{rmk}

Cases of particular interest of Theorem \ref{t:dc} include right-angled Artin groups and graph products of finite groups. More generally, let us note two special cases of pairs of $(G,X)$ satisfying Definition \ref{d:rp}:
\begin{enumerate}[(i)]
\item Let $G$ be virtually nilpotent, and $X$ be a finite generating set with $s_{G,X}(t)$ rational: in particular, this holds whenever $G$ is virtually abelian \cite{benson} and for $G = H_1$, the integral Heisenberg group \cite{duchin}. It was shown that by Wolf \cite{wolf} that if $G$ is virtually nilpotent then it has polynomial growth (by Gromov's Theorem \cite{gromov81}, the converse is also true), and so part \eqref{i:drp-sc} of Definition \ref{d:rp} holds trivially by bounding growth of centralisers by the growth of $G$ itself.
\item Let $G$ be a torsion-free hyperbolic group, and $X$ be any finite generating set. Cannon \cite{cannon} and Gromov \cite[Theorem 8.5.N]{gromov87} have shown that hyperbolic groups have rational growth with respect to any generating set, and all infinite-order elements have virtually cyclic centralisers. Moreover, for any torsion-free hyperbolic group $G$ with a finite generating set $X$, there is a constant $P > 0$ such that $|C_G(g) \cap B_{G,X}(n)| \leq Pn$ for all $n \geq 1$ and all non-trivial $g \in G$: see the proof of Theorem 1.7 in \cite{amv} for details and references.
\end{enumerate}

The paper is structured as follows. Section \ref{s:growth} applies to all infinite groups with rational spherical growth series and is dedicated to a proof of Theorem \ref{t:growth}. Section \ref{s:dc} is used to prove Theorem \ref{t:dc}.

\begin{ack}
The author would like to give special thanks to his Ph.D.\ supervisor, Armando Martino, without whose help and guidance this paper would not have been possible. He would also like to thank Yago Antol\'in, Charles Cox and Enric Ventura for valuable discussions and advice, as well as Ashot Minasyan and anonymous referees for their comments on this manuscript. Finally, the author would like to give credit to Gerald Williams for a question which led to generalising a previous version of Theorem \ref{t:dc}. The author was funded by EPSRC Studentship 1807335.
\end{ack}

\section{Groups with rational growth series}
\label{s:growth}

This section provides a proof of Theorem \ref{t:growth}. Let $G$ be an infinite group, and suppose that the growth series of $G$ with respect to a finite generating set $X$ is a rational function. In particular, the spherical growth series is
\begin{equation*}
s(t) = s_{G,X}(t) = \sum_{n=0}^\infty \sss(n)t^n = \frac{p(t)}{q(t)}
\end{equation*}
where $\sss(n) = \sss_G(n) = \sss_{G,X}(n) := |S_{G,X}(n)|$, and
\begin{equation*}
q(t) = q_0 t^c \prod_{i=1}^r (1-\lambda_i t)^{\alpha_i+1} \quad \text{and} \quad p(t) = p_0 t^{\tilde{c}} \prod_{i=1}^{\tilde{r}} (1-\tilde{\lambda}_i t)^{\tilde{\alpha}_i+1}
\end{equation*}
are non-zero polynomials with no common roots (and so either $c = 0$ or $\tilde{c} = 0$), with $\alpha_i, \tilde{\alpha}_i \in \Z_{\geq 0}$ for all $i$. Since the series $(\sss(n))_{n=0}^\infty$ grows at most exponentially, $s(t)$ is analytic (and so continuous) at $0$, hence one has
\begin{equation*}
1 = \sss(0) = \lim_{t \to 0} s(t) = \frac{p_0}{q_0} \lim_{t \to 0} t^{\tilde{c}-c}
\end{equation*}
and so $c = \tilde{c}$ and $p_0 = q_0$. Thus $c = \tilde{c} = 0$ and, without loss of generality, $q_0 = p_0 = 1$.

Coefficients of such a series are described in \cite[Lemma 1]{lee}; in particular, it follows that
\begin{equation}
\label{e:exprss}
\sss(n) = \sum_{i=1}^r \sum_{j=0}^{\alpha_i} b_{i,j} n^j \lambda_i^n
\end{equation}
for $n$ large enough, with $b_{i,\alpha_i} \neq 0$ for all $i$.

Now consider the terms of \eqref{e:exprss} that give a non-negligible contribution to $\sss(n)$ for large $n$. In particular, one may assume without loss of generality that
\begin{equation*}
\lambda := |\lambda_1| = |\lambda_2| = \cdots = |\lambda_{\tilde{k}}| > |\lambda_{\tilde{k}+1}| \geq |\lambda_{\tilde{k}+2}| \geq \cdots \geq |\lambda_r|
\end{equation*}
for some $\tilde{k} \leq r$ and that
\begin{equation*}
\alpha := \alpha_1 = \alpha_2 = \cdots = \alpha_k > \alpha_{k+1} \geq \alpha_{k+2} \geq \cdots \geq \alpha_{\tilde{k}}
\end{equation*}
for some $k \leq \tilde{k}$. Note that one must have $\lambda \geq 1$: otherwise the radius of convergence of $s(t)$ is $\lambda^{-1}>1$ and so the series $\sum_n \sss(n)$ converges, contradicting the fact that $G$ is infinite.

For $n \in \Z_{\geq 0}$, define
\begin{equation*}
c_n = \sum_{j=1}^k b_{j,\alpha} \exp(i \varphi_j n)
\end{equation*}
where $\lambda_j = \lambda \exp(i \varphi_j)$ for some $\varphi_j \in (-\pi,\pi]$, for $1 \leq j \leq k$. It follows that
\begin{equation}
\label{e:ancn}
\sss(n) = n^\alpha \lambda^n (c_n+o(1))
\end{equation}
as $n \to \infty$. In particular, since $\sss(n) \in (0,\infty) \subseteq \R$ for all $n$, it follows that
\begin{equation}
\label{e:lims}
\liminf_{n \to \infty} \Re(c_n) \geq 0 \quad \text{and} \quad \lim_{n \to \infty} \Im(c_n) = 0.
\end{equation}

It is clear that
\begin{equation*}
\limsup_{n \to \infty} \frac{\sss(n)}{n^\alpha \lambda^n} \leq \sum_{j=1}^k |b_{j,\alpha}|,
\end{equation*}
which shows existence of the constant $D$ in Theorem \ref{t:growth}; in order to prove the Proposition, it is enough to show that $\liminf_{n \to \infty} \sss(n)/(n^\alpha \lambda^n) > 0$. However, this bound does not follow solely from the fact that $s(t)$ is a rational function: see Example \ref{ex:liminf} \eqref{i:ex-sm} at the end of this section.
\begin{rmk}
\label{r:submult}
Clearly, for any $n_1,n_2 \geq 0$, if $g \in G$ has $|g|_X = n_1+n_2$ (respectively $|g|_X \leq n_1+n_2$), then one can write $g = g_1g_2$ where $|g_j|_X = n_j$ (respectively $|g_j|_X \leq n_j$) for $j \in \{1,2\}$. This gives injections $S(n_1+n_2) \to S(n_1) \times S(n_2)$ and $B(n_1+n_2) \to B(n_1) \times B(n_2)$ by mapping $g \mapsto (g_1,g_2)$. In particular, it follows that
\begin{equation*}
\sss(n_1+n_2) \leq \sss(n_1)\sss(n_2) \quad \text{and} \quad \bb(n_1+n_2) \leq \bb(n_1)\bb(n_2)
\end{equation*}
for any $n_1,n_2 \in \Z_{\geq 0}$. This property is called \emph{submultiplicativity} of sphere and ball sizes in $G$.
\end{rmk}

The aim is now to show that submultiplicativity of the sequence $(\sss(n))_{n=0}^\infty$, together with rationality of $s(t)$, implies the conclusion of Theorem \ref{t:growth}. As the $b_{j,\alpha}$ are non-zero and the $\varphi_j$ are distinct, given \eqref{e:lims} the following result seems highly likely:

\begin{lem}
\label{l:rd}
The numbers $c_n$ are real, and for some constant $\delta > 0$, the set
\begin{equation*}
E_\delta := \{ n \in \Z_{\geq 0} \mid c_n \geq \delta \}
\end{equation*}
is relatively dense in $[0,\infty)$, i.e.\ the inclusion $E_\delta \hookrightarrow [0,\infty)$ is a $(1,K)$-quasi-isometry for some $K \geq 0$.
\end{lem}

However, the author has been unable to come up with a straightforward proof of Lemma \ref{l:rd} without using some additional theory on `quasi-periodicity' of the sequence $(c_n)_{n=0}^\infty$. Before giving a proof, let us deduce Theorem \ref{t:growth} from Lemma \ref{l:rd}.

Assuming Lemma \ref{l:rd}, one can find $N \in \Z_{\geq 1}$ such that for all $n$, there exists a $\beta = \beta_n \in \{0,\ldots,N\}$ with $c_{n+\beta} \geq \delta$. Define
\begin{equation*}
R := \max \{ \lambda^{-\beta} \sss(\beta) \mid 0 \leq \beta \leq N \},
\end{equation*}
and let $M \in \Z_{\geq 1}$ be such that for all $n \geq M$, one has
\begin{equation*}
\sss(n) \geq n^\alpha \lambda^n \left( c_n - \frac{\delta}{2} \right)
\end{equation*}
(such an $M$ exists by \eqref{e:ancn}). Then submultiplicativity of sphere sizes implies that for all $n \geq M$,
\begin{equation*}
\begin{split}
\frac{\delta}{2} (n+\beta_n)^\alpha \lambda^{n+\beta_n} &\leq \left( c_{n+\beta_n} - \frac{\delta}{2} \right) (n+\beta_n)^\alpha \lambda^{n+\beta_n} \\ &\leq \sss(n+\beta_n) \leq \sss(n)\sss(\beta_n) \leq \sss(n) R \lambda^{\beta_n}.
\end{split}
\end{equation*}
It follows that
\begin{equation*}
\sss(n) \geq \frac{\delta}{2R} (n+\beta_n)^\alpha \lambda^n \geq \frac{\delta}{2R} n^\alpha \lambda^n
\end{equation*}
for $n \geq M$, showing that
\begin{equation*}
\liminf_{n \to \infty} \frac{\sss(n)}{n^\alpha \lambda^n} \geq \frac{\delta}{2R} > 0,
\end{equation*}
which shows existence of the constant $C > 0$ in Theorem \ref{t:growth}. Thus in order to prove Theorem \ref{t:growth} it is now enough to prove Lemma \ref{l:rd}.

\begin{proof}[Proof of Lemma \ref{l:rd}]
To prove the Lemma, one may employ a digression into a certain class of functions from $\R$ to $\C$, called `uniformly almost periodic functions'. The theory for these functions is presented in a book by Besicovitch \cite{besicovitch}.

Let $f: \R \to \C$ be a function. Given $\varepsilon > 0$, define the set $E(f,\varepsilon) \subseteq \R$ to be the set of all numbers $\tau \in \R$ (called the \emph{translation numbers} for $f$ belonging to $\varepsilon$) such that
\begin{equation*}
\sup_{x \in \R} |f(x+\tau)-f(x)| \leq \varepsilon.
\end{equation*}
The function $f$ is said to be \emph{uniformly almost periodic} (\emph{u.~a.~p.}) if, for any $\varepsilon > 0$, the set $E(f,\varepsilon)$ is relatively dense in $\R$, i.e.\ the inclusion $E(f,\varepsilon) \hookrightarrow \R$ is a $(1,K)$-quasi-isometry for some $K \geq 0$. It is easy to see that any periodic function is u.~a.~p., and that every continuous u.~a.~p.\ function is bounded.

Now note that the function
\begin{equation*}
\begin{split}
c: \R &\to \C \\ t &\mapsto \sum_{j=1}^k b_{j,\alpha} \exp(i \varphi_j t)
\end{split}
\end{equation*}
is a sum of continuous periodic functions, and so is a continuous u.~a.~p.\ function by \cite[Section 1.1, Theorem 12]{besicovitch}. By definition, $c_n = c(n)$ for any $n \in \Z_{\geq 0}$.

The aim is to show that the function $\bar{c}: t \mapsto c(\lfloor t \rfloor)$ is also u.~a.~p. For this, note that $c$ is everywhere differentiable and the derivative $c'(t)$ is a sum of continuous periodic functions, so is continuous and u.~a.~p.\ -- in particular, it is bounded, by some $R > 0$, say. For a given $\varepsilon \in (0,R)$, set a constant $M := \varepsilon / \left( 2\sin\left(\frac{\pi\varepsilon}{2R}\right) \right)$ and define $f: \R \to \R$ by $f(t) = M\sin(\pi t)$. It is easy to check that
\begin{equation}
\label{e:earoundz}
E\left(f,\frac{\varepsilon}{2}\right) \subseteq \bigcup_{n \in \Z} \left[ n-\frac{\varepsilon}{2R}, n+\frac{\varepsilon}{2R} \right].
\end{equation}

For any $\tau \in \R$, define $n_\tau = \left\lfloor \tau+\frac{1}{2} \right\rfloor \in \Z$ to be the nearest integer to $\tau$. Pick $\tau \in E\left(f,\frac{\varepsilon}{2}\right) \cap E\left(c,\frac{\varepsilon}{2}\right)$ -- then $|c(x+\tau)-c(x)| \leq \frac{\varepsilon}{2}$ for all $x \in \R$, and, by \eqref{e:earoundz}, $|\tau-n_\tau| \leq \frac{\varepsilon}{2R}$, so in particular $|c(x+\tau)-c(x+n_\tau)| \leq \frac{\varepsilon}{2}$ for all $x \in \R$ by the choice of $R$. Thus $|c(x+n_\tau)-c(x)| \leq \varepsilon$ for all $x \in \R$, i.e.\ $n_\tau \in E(c,\varepsilon)$.

But by \cite[Section 1.1, Theorem 11]{besicovitch}, the set $E\left(f,\frac{\varepsilon}{2}\right) \cap E\left(c,\frac{\varepsilon}{2}\right)$ is relatively dense, hence (by the previous paragraph) so is the set $E(c,\varepsilon) \cap \Z$. However, for any $n \in E(c,\varepsilon) \cap \Z$ and any $x \in \R$ one has
\begin{equation*}
|\bar{c}(x+n)-\bar{c}(x)| = |c(\lfloor x+n \rfloor) - c(\lfloor x \rfloor)| = |c(\lfloor x \rfloor+n) - c(\lfloor x \rfloor)| \leq \varepsilon
\end{equation*}
and so $E(c,\varepsilon) \cap \Z \subseteq E(\bar{c},\varepsilon) \cap \Z$. It follows that $E(\bar{c},\varepsilon) \cap \Z$ is relatively dense (and so the function $\bar{c}: t \mapsto c(\lfloor t \rfloor)$ is u.~a.~p.).

Now recall that \eqref{e:lims} provides constraints for limits of sequences $(\Re(c_n))$ and $(\Im(c_n))$: namely,
\begin{equation} \label{e:lims2}
\liminf_{n \to \infty} \Re(c_n) \geq 0 \quad \text{and} \quad \lim_{n \to \infty} \Im(c_n) = 0.
\end{equation}
It is easy to see that $c_n \in \R_{\geq 0}$ for all $n$: indeed, if either $\Re(c_n) = -\delta < 0$ or $|\Im(c_n)| = \delta > 0$ for some $n$ then the fact that the set $E(\bar{c},\delta/2) \cap \Z$ is relatively dense contradicts \eqref{e:lims2}. Similarly, if $c_N > 0$ for some $N$ then the set $E(\bar{c},\delta) \cap \Z$ is a relatively dense set contained in the set $\{ n \in \Z \mid c(n) \geq \delta \}$, where $\delta = c_N/2$. To prove Lemma \ref{l:rd} it is therefore enough to show that the sequence $(c_n)_{n=0}^\infty$ is not identically zero.

Now recall that the sequence $(c_n)$ is defined by
\begin{equation*}
c_n = \sum_{j=1}^k b_{j,\alpha} \exp(i\varphi_jn),
\end{equation*}
and suppose for contradiction that $c_n = 0$ for all $n \in \Z_{\geq 0}$, and in particular for $0 \leq n \leq k-1$.
This is the same as saying that $Mv = 0$, where
\begin{equation*}
M = \begin{pmatrix}
1 & 1 & \cdots & 1 \\
\exp(i\varphi_1) & \exp(i\varphi_2) & \cdots & \exp(i\varphi_k) \\
\vdots & \vdots & \ddots & \vdots \\
\exp(i\varphi_1)^{k-1} & \exp(i\varphi_2)^{k-1} & \cdots & \exp(i\varphi_k)^{k-1}
\end{pmatrix}
\end{equation*}
and
\begin{equation*}
v = \begin{pmatrix} b_{1,\alpha-1} \\ b_{2,\alpha-1} \\ \vdots \\ b_{k,\alpha-1} \end{pmatrix}.
\end{equation*}
Thus $M$ has a zero eigenvalue and so $\det M = 0$. But $M^t$ is a Vandermonde matrix with pairwise distinct rows, so $\det M \neq 0$. This gives a contradiction which completes the proof.
\end{proof}

\begin{rmk}
A stronger conclusion of Theorem \ref{t:growth} holds if in addition $s_{G,X}(t)$ is a \emph{positive} rational function, i.e.\ it is contained in the smallest sub-semiring of $\C(t)$ containing the semiring $\Z_{\geq 0}[t]$ and closed under quasi-inversion, $f(t) \mapsto (1-f(t))^{-1}$ (for $f(t) \in \C(t)$ with $f(0) = 0$). This is the case in particular if there exists a language $\mathcal{L}$ in $(X \cup X^{-1})^\ast$ that is regular (i.e.\ recognised by a finite state automaton), the monoid homomorphism $\Phi: \mathcal{L} \to G$ extending the inclusion $X \cup X^{-1} \hookrightarrow G$ is a bijection, and $\mathcal{L}$ consists only of geodesic words in the Cayley graph of $G$ with respect to $X$, i.e.\ the length of any word $l \in \mathcal{L}$ is $|\Phi(l)|_X$. If $s_{G,X}(t)$ is a positive rational function, then the numbers $\varphi_j$ above are in fact rational multiples of $\pi$ \cite{berstel}, and as a consequence the sequence $(c_n)$ is periodic. %This strengthens the conclusion of Theorem \ref{t:growth}: instead of saying that the sequence $\left( \frac{|S_{G,X}(n)|}{n^\alpha \lambda^n} \right)$ has (non-infinite) upper and (non-zero) lower bounds, it now says that it has finitely many (non-zero and non-infinite) limit points.

However, the author has not been able to find a reason why the function $s_{G,X}(t)$, in case it is rational, must also be positive. In particular, one can find pairs $(G,X)$ such that $s_{G,X}(t)$ is rational but there are no regular languages $\mathcal{L}$ as above, and one can even find groups $G$ such that this holds for $(G,X)$ for any generating set $X$. For instance, it can be shown that growth of the 2-step nilpotent Heisenberg group
\begin{equation*}
G = H_3 = \langle a,b,c \mid [a,b]=c, [a,c]=[b,c]=1 \rangle
\end{equation*}
is rational with respect to any generating set \cite[Theorem 1]{duchin}, but there are no languages $\mathcal{L}$ as above when $G$ is a 2-step nilpotent group that is not virtually abelian \cite[Corollary 3]{stoll95}.
\end{rmk}

It is easy to check that the conclusion of Theorem \ref{t:growth} implies that
\begin{equation} \label{e:ballbounds}
\liminf_{n \to \infty} \frac{|B_{G,X}(n)|}{n^{\hat\alpha} \lambda^n} > 0 \qquad \text{and} \qquad \limsup_{n \to \infty} \frac{|B_{G,X}(n)|}{n^{\hat\alpha} \lambda^n} < \infty,
\end{equation}
where $\hat\alpha = \alpha+1$ if $\lambda = 1$ and $\hat\alpha = \alpha$ otherwise. Asymptotics similar to these have been obtained for nilpotent groups, even without the condition on rational growth. In particular, in \cite{pansu} Pansu showed that given a nilpotent group $G$ with a finite generating set $X$, there exists $\hat\alpha \in \Z_{\geq 0}$ such that $\frac{|B_{G,X}(n)|}{n^{\hat\alpha}} \to C$ as $n \to \infty$ for some $C > 0$. Moreover, in \cite{stoll96} Stoll calculates the constant $C$ for certain $2$-step nilpotent groups $G$ explicitly to show that the corresponding growth series $s_{G,X}(t)$ \emph{cannot} be rational.
However, in general -- for groups that are not virtually nilpotent -- one cannot expect $\limsup$ and $\liminf$ in \eqref{e:ballbounds} to be equal, as the hyperbolic group $C_2 * C_3$ shows: see \cite[\S 3]{griha}.

Finally, note that the same proof indeed shows a more general result:
\begin{thm} \label{t:growth2}
Let $(a_n)_{n=0}^\infty$ be a submultiplicative sequence of numbers in $\Z_{\geq 1}$ such that $s(t) = \sum a_nt^n$ is a rational function. Then there exist constants $\alpha \in \Z_{\geq 0}$, $\lambda \in [1,\infty)$ and $D > C > 0$ such that for all $n \geq 1$,
\begin{equation*}
Cn^\alpha \lambda^n \leq a_n \leq Dn^\alpha \lambda^n.
\end{equation*}
\end{thm}

The example below shows that both submultiplicativity and rationality are necessary requirements.
\begin{ex} \label{ex:liminf} \begin{enumerate}[(i)]
\item \label{i:ex-sm} Let $$p(t) = 1 + 12t^2 - 16t^3$$ and $$q(t) = (1-t)(1-2t)(1-2\omega t)(1-2\bar{\omega}t),$$ where $\omega$ is a $6^\text{th}$ primitive root of unity. Let $s(t)$, $(a_n)$, $\lambda$, $\alpha$ and $(c_n)$ be as above. Then $\lambda = 2$ and $\alpha = 0$, and \cite[Lemma 1]{lee} can be used to calculate
\begin{equation*}
a_n = c_n2^n + 1
\end{equation*}
where
\begin{equation*}
c_n = 4 - 2\omega^n - 2\bar{\omega}^n = \begin{cases} 0, & n \equiv 0 \pmod{6}, \\ 2, & n \equiv \pm 1 \pmod{6}, \\ 6, & n \equiv \pm 2 \pmod{6}, \\ 8, & n \equiv 3 \pmod{6}. \end{cases}
\end{equation*}
But as $c_n = 0$ for infinitely many values of $n$, one has $$\liminf_{n \to \infty} a_n/(n^\alpha \lambda^n) = 0.$$ Note that in this case $a_7 = 257 > 5 = a_1a_6$, so the sequence $(a_n)$ is not submultiplicative.
\item \label{i:ex-rat} For $n \geq 0$, let $a_n = 2^{b(n)}$, where $b(n)$ is the sum of digits in the binary representation of $n$. Then $(a_n)$ is a submultiplicative sequence, but $\sum a_n t^n$ is not a rational function. For each $n \geq 0$, one has $a_{2^n-1} = 2^n$ and $a_{2^n} = 2$. Thus $$\liminf_{n \to \infty} \frac{a_n}{n} \leq \liminf_{n \to \infty} \frac{2}{2^n} = 0$$ and $$\limsup_{n \to \infty} a_n \geq \limsup_{n \to \infty} 2^n = \infty,$$ so $(a_n)$ does not satisfy the conclusion of Theorem \ref{t:growth2} for any $\lambda \geq 1$ and $\alpha \in \Z_{\geq 0}$.
\end{enumerate}
\end{ex}

\section{Degree of commutativity}
\label{s:dc}

The aim of this section is to prove Theorem \ref{t:dc}. For this, let $\Gamma$ be a finite simple graph and for each $v \in V(\Gamma)$, let $(\hh(v),X(v))$ be a rational pair with small centralisers (see Definition \ref{d:rp}). To simplify notation, suppose in addition that the sets $X(v)$ are symmetric and do \emph{not} contain the identity $1 \in \hh(v)$: clearly this does not affect the results. Suppose in addition that $G = G(\Gamma,\hh)$ is a group of exponential growth. One thus aims to show that $\dc_{X}(G) = 0$, where $X = \bigsqcup_{v \in V(\Gamma)} X(v)$.

\subsection{Preliminaries}
\label{ss:prelim}

This subsection collects the terminology and preliminary results used in the proof of Theorem \ref{t:dc}.

Let $\ell_n: X^\ast \to \Z_{\geq 0}$ be the \emph{normal form length} function ($n$ in $\ell_n$ stands for `normal'): for $w \in X^\ast$, set $\ell_n(w) := m$ where $m$ is the minimal integer for which $w \equiv w_1w_2 \cdots w_m$ as words, where $w_i \in X(v_i)^\ast$ for some $v_i \in V(\Gamma)$. Moreover, let $\ell_w: X^\ast \to \Z_{\geq 0}$ be the \emph{word length} function ($w$ in $\ell_w$ stands for `word'), i.e.\ let $\ell_w(w)$ be the number of letters in $w \in X^\ast$.

The following result says that given any word $w \in X^\ast$ representing $g \in G$, there is a simple algorithm to transform it into a word $\hat{w}$ representing $g$ with $\ell_n(\hat{w})$ or $\ell_w(\hat{w})$ small. This follows quite easily from a result of Green \cite{green}.

\begin{prop} \label{p:wbeh}
Let $\ell: X^\ast \to \Z_{\geq 0}$ be either $\ell = \ell_n$ or $\ell = \ell_w$. Let $w \in X^\ast$ be a word representing an element $g \in G$, and let $\hat{w}$ be a word representing $g$ with $(\ell(\hat{w}),\ell_w(\hat{w}))$ minimal (in the lexicographical ordering) among such words. Then $\hat{w}$ can be obtained from $w$ by applying a sequence of moves of two types:
\begin{enumerate}[(i)]
\item \label{i:move1} for some $w_u \in X(u)^\ast$ and $w_v \in X(v)^\ast$ with $\{ u,v \} \in E(\Gamma)$, replacing a subword $w_uw_v$ with $w_vw_u$;
\item \label{i:move2} for some $v \in V(\Gamma)$ and some subword $w_1 \in X(v)^\ast$, replacing the subword $w_1$ with a word $w_0 \in X(v)^\ast$ representing the same element in $\hh(v)$, such that $\ell_w(w_0) \leq \ell_w(w_1)$.
\end{enumerate}
\end{prop}

\begin{proof}
Suppose first that $\ell = \ell_n$, and let $\hat{w} \equiv w_1 \cdots w_m$, where $w_i \in X(v_i)^\ast$ for some $v_i \in V(\Gamma)$ and $m = \ell_n(w)$. In \cite[Theorem 3.9]{green}, Green showed that by using moves \eqref{i:move1} and \eqref{i:move2} we can transform $w$ into a word $\hat{w}' \equiv w_1' \cdots w_m'$ where $w_i' \in X(v_i)^\ast$ and $w_i$, $w_i'$ represent the same element of $\hh(v)$. Notice that we have $\ell_w(w_i) \leq \ell_w(w_i')$ for each $i$: otherwise, existence of the word $w_1 \cdots w_{i-1} w_i' w_{i+1} \cdots w_m$ would contradict the minimality of $\hat{w}$. Thus a sequence of moves \eqref{i:move2} allows us to transform $\hat{w}'$ into $\hat{w}$, as required.

Suppose now that $\ell = \ell_w$. Let $\hat{w}_n \in X^*$ be a word representing $g$ with $(\ell_n(\hat{w}_n),\ell_w(\hat{w}_n))$ minimal among all such words. Then the result for $\ell = \ell_n$ says that $\hat{w}$ can be transformed into $\hat{w}_n$ by using the moves \eqref{i:move1}--\eqref{i:move2}. Notice that if $w' \in X^\ast$ is obtained from $w \in X^\ast$ by applying move \eqref{i:move1} or \eqref{i:move2}, then $\ell_w(w') \leq \ell_w(w)$, and if the equality holds then there exists a move that transforms $w'$ back into $w$. By definition of $\hat{w}$, no moves strictly decreasing the word length are used when transforming $\hat{w}$ to $\hat{w}_n$, and so there exists a sequence of moves transforming $\hat{w}_n$ into $\hat{w}$ as well. Thus we may apply moves \eqref{i:move1}--\eqref{i:move2} to obtain $\hat{w}_n$ from $w$ and subsequently $\hat{w}$ from $\hat{w}_n$, as required.
\end{proof}

Note that it follows from the proof of Proposition \ref{p:wbeh} that minimal values of $\ell_n(w)$ and $\ell_w(w)$ can be obtained simultaneously. This justifies the following:
\begin{defn}
For $g \in G$, define a \emph{normal form} of $g$ to be a word $w \in X^\ast$ with both $\ell_n(w)$ and $\ell_w(w)$ minimal (so that $\ell_w(w) = |g|_X$). Write $w = w_1 w_2 \cdots w_n$ for $w_i \in X$, and define the \emph{support} of $g$ as
\begin{equation*}
\supp(g) := \{ v \in V(\Gamma) \mid w_i \in X(v) \text{ for some } i \};
\end{equation*}
by Proposition \ref{p:wbeh} this does not depend on the choice of $w$.
\end{defn}

Now suppose for contradiction that $\dc_X(G) > 0$. That means that for some constant $\varepsilon > 0$, one has
\begin{equation} \label{e:pdef}
\sum_{g \in B(n)} \frac{|C_G(g) \cap B(n)|}{\bb(n)^2} \geq \varepsilon
\end{equation}
for infinitely many values of $n$, where $C_G(g)$ denotes the centraliser of an element $g \in G$, and $\bb(n) = \bb_G(n) = \bb_{G,X}(n) := |B_{G,X}(n)|$.

In the proof certain conjugates of elements in $G$ will be considered. In particular, let $g \in G$, and pick a conjugate $\tilde{g} \in G$ of $g$ such that $g = p_g^{-1} \tilde{g} p_g$ with $|g| = 2|p_g| + |\tilde{g}|$ and such that $|\tilde{g}|$ is minimal subject to this. If $p_g = 1$, then $g$ is called \emph{cyclically reduced}; hence $\tilde{g}$ is cyclically reduced. Note that being cyclically reduced is a weaker condition than being cyclically normal in the sense of \cite{barkauskas}.

For any subset $A \subseteq V(\Gamma)$, let $G_A$ denote $G(\Gamma(A),\hh|_A)$, where $\Gamma(A)$ is the full subgraph of $\Gamma$ spanned by $A$. These will be viewed as subgroups (called the \emph{special subgroups}) of $G$. One may also define the \emph{link} of $A$ to be
\begin{equation*}
\link A = \{ u \in V(\Gamma) \mid (u,v) \in E(\Gamma) \text{ for all } v \in A \}.
\end{equation*}

Before carrying on with the proof, consider the sequence $(d_n)_{n=0}^\infty$ where
\begin{equation*}
d_n := \frac{|\{ (x,y) \in B_{G,X}(n)^2 \mid [x,y]=1 \}|}{\bb_{G,X}(n)^2}.
\end{equation*}
One aims to show that $d_n \to 0$ as $n \to \infty$. Note that for many groups of exponential growth, including all the non-elementary hyperbolic groups \cite{amv}, the sequence $(d_n)_{n=0}^\infty$ converges to zero exponentially fast. However, the following example shows that this is not always the case for graph products. The result of Theorem \ref{t:dc} may be therefore more delicate than one might think.

\begin{ex}
Suppose $\Gamma$ is a complete bipartite graph $K_{k,k}$, i.e.\ $\Gamma$ has vertex set 
\begin{equation*}
V(\Gamma) = \{ u_1,\ldots,u_k,v_1,\ldots,v_k \}
\end{equation*}
and edge set 
\begin{equation*}
E(\Gamma) = \{ \{u_i,v_j\} \mid 1 \leq i,j \leq k \},
\end{equation*}
and let $\hh(u) \cong \Z$ with generators $X(u) = \{ x_u, x_u^{-1} \}$ for each $u \in V(\Gamma)$. In this case one has $G(\Gamma,\hh) \cong F_k \times F_k$ (direct product of two free groups of rank $k$) and so one can calculate sphere sizes in $G(\Gamma,\hh)$ and its special subgroups easily. Note that clearly (by the definition of link) every element of $G_A \leq G$ commutes with every element of $G_{\link A} \leq G$. Now consider the case where $A = \{ u_1,\ldots,u_k \}$ and so $\link A = \{ v_1,\ldots,v_k \}$. It follows that
\begin{equation*}
\{ (x,y) \in B(n)^2 \mid [x,y]=1 \} \supseteq B_{G_A}(n) \times B_{G_{\link A}}(n).
\end{equation*}
An explicit computation shows that
\begin{equation*}
\bb_{G_A}(n) = \bb_{G_{\link A}}(n) = \frac{k(2k-1)^n-1}{k-1}
\end{equation*}
and
\begin{equation*}
\bb_G(n) = \frac{2k^2n(2k-1)^n}{(k-1)(2k-1)} + e_1 (2k-1)^n + e_2
\end{equation*}
where $e_1 = e_1(k)$ and $e_2 = e_2(k)$ are some constants. It follows that
\begin{equation*}
d_n \geq \frac{\bb_{G_A}(n) \bb_{G_{\link A}}(n)}{\bb_G(n)^2} \sim \left( \frac{2k-1}{2kn} \right)^2
\end{equation*}
as $n \to \infty$. In particular, the sequence $(d_n)_{n=0}^\infty$ converges to zero only at a polynomial rate for $G = G(\Gamma,\hh)$.
\end{ex}

The proof of Theorem \ref{t:dc} is based on the fact that if \eqref{e:pdef} held for infinitely many $n$ then there would exist a subset $A \subseteq V(\Gamma)$ such that the growth of both $G_A$ and $G_{\link A}$ would be comparable to that of $G$. More precisely, the outline of the proof is as follows:
\begin{enumerate}[(i)]
\item finding such a subset $A \subseteq V(\Gamma)$ and showing that $G_A$ is not negligible in $G$, i.e.\ $\frac{\bb_{G_A}(n)}{\bb_G(n)} \nrightarrow 0$ as $n \to \infty$ (subsection \ref{ss:getA});
\item finding a collection $\mathcal{H}$ of subgroups of $G$ having (uniformly) polynomial growth such that, for all $H \in \mathcal{H}$, $G_{\link A} \times H$ is a subgroup of $G$ and $\frac{|(G_{\link A} \times H) \cap B_G(n)|}{\bb_G(n)}$ is uniformly bounded below as $n \to \infty$ (subsection \ref{ss:centr});
\item using the embedding $G_A \times G_{\link A} \subseteq G$ and Theorem \ref{t:growth} to obtain a contradiction (subsection \ref{ss:prod}).
\end{enumerate}

\subsection{A non-negligible special subgroup}
\label{ss:getA}

Note that \eqref{e:pdef} can be rewritten as
\begin{equation} \label{e:s11}
\sum_{A \subseteq V(\Gamma)} \sum_{\substack{g \in B(n) \\ \supp(\tilde{g}) = A}} \frac{|C_G(g) \cap B(n)|}{\bb(n)^2} \geq \varepsilon
\end{equation}
and so \eqref{e:s11} holds for infinitely many $n$. But as $\Gamma$ is finite, there are only $2^{|V(\Gamma)|} < \infty$ subsets of $V(\Gamma)$, thus in particular there exists a subset $A \subseteq V(\Gamma)$ such that
\begin{equation} \label{e:s12}
\sum_{\substack{g \in B(n) \\ \supp(\tilde{g}) = A}} \frac{|C_G(g) \cap B(n)|}{\bb(n)^2} \geq 2^{-|V(\Gamma)|}\varepsilon
\end{equation}
holds for infinitely many $n$. One may restrict the subset of elements $g \in G$ considered even further:

\begin{lem} \label{l:pgsmall} There exist constants $\tilde\varepsilon > 0$ and $s \in \Z_{\geq 0}$ such that
\begin{equation*}
\sum_{\substack{g \in B(n) \\ \supp(\tilde{g}) = A \\ |p_g| \leq s}} \frac{|C_G(g) \cap B(n)|}{\bb(n)^2} \geq \tilde\varepsilon
\end{equation*}
for infinitely many $n$.
\end{lem}
\begin{proof}
As $G$ has rational spherical growth series by \cite{chiswell}, Theorem \ref{t:growth} says that there exist constants $\alpha \in \Z_{\geq 0}$, $\lambda \geq 1$, $C = C_G > 0$ and $D = D_G > C$ such that
\begin{equation} \label{e:Sbound}
C n^\alpha \lambda^n \leq \sss(n) \leq D n^\alpha \lambda^n
\end{equation}
for all $n \geq 1$. As it is also assumed that $G$ has exponential growth, one has $\lambda > 1$. It is easy to show that in this case
\begin{equation} \label{e:Bbound}
C n^\alpha \lambda^n < \bb(n) < \frac{D\lambda}{\lambda-1} n^\alpha \lambda^n
\end{equation}
for all $n \geq 1$.

Now one can bound the number of terms in \eqref{e:s12} corresponding to elements $g \in G$ with $|p_g|$ large (even without requiring $\supp(\tilde{g}) = A$). Indeed, as any $g \in G$ can be written as $g = p_g^{-1} \tilde{g} p_g$ with $|g| = 2|p_g| + |\tilde{g}|$, \eqref{e:Sbound} and \eqref{e:Bbound} imply
\begin{equation}
\begin{split} \label{e:s21}
\frac{1}{\bb(n)} \sum_{\substack{g \in B(n) \\ |p_g| > s}} &\frac{|C_G(g) \cap B(n)|}{\bb(n)} \leq \frac{|\{ g \in B(n) \mid |p_g| > s \}|}{\bb(n)} \leq \sum_{i=s+1}^{\left\lfloor \frac{n}{2} \right\rfloor} \frac{\sss(i) \bb(n-2i)}{\bb(n)} \\ &\leq \frac{D}{C} \left( \frac{1}{2} \right)^\alpha \lambda^{-\frac{n}{2}} + \frac{D^2 \lambda}{C(\lambda-1)} \sum_{i=s+1}^{\left\lfloor \frac{n-1}{2} \right\rfloor} \left( \frac{i(n-2i)}{n} \right)^\alpha \lambda^{-i}.
\end{split}
\end{equation}
The first term of the sum above clearly tends to zero as $n \to \infty$, and the second term is bounded above by the infinite sum $\sum_{i=s+1}^\infty i^\alpha \lambda^{-i}$, which tends to zero as $s \to \infty$ since the series $\sum_i i^\alpha \lambda^{-i}$ converges. Hence there exists a value of $s \in \Z_{\geq 0}$ which ensures that the right hand side in \eqref{e:s21} is less than $2^{-|V(\Gamma)|-1}\varepsilon$ for $n$ large enough. This means that
\begin{equation*}
\sum_{\substack{g \in B(n) \\ \supp(\tilde{g}) = A \\ |p_g| \leq s}} \frac{|C_G(g) \cap B(n)|}{\bb(n)^2} \geq 2^{-|V(\Gamma)|-1}\varepsilon
\end{equation*}
for infinitely many $n$, so setting $\tilde\varepsilon := 2^{-|V(\Gamma)|-1}\varepsilon$ completes the proof.
\end{proof}

Now note that one may write
\begin{equation*} \label{e:s31}
\begin{split}
\sum_{\substack{g \in B(n) \\ \supp(\tilde{g}) = A \\ |p_g| \leq s}} &\frac{|C_G(g) \cap B(n)|}{\bb(n)^2} \leq \frac{|\{ g \in B(n) \mid \supp(\tilde{g}) = A, |p_g| \leq s \}|}{\bb(n)} \\ &\times \max \left\{ \frac{|C_G(g) \cap B(n)|}{\bb(n)} \:\middle|\: g \in B(n), \supp(\tilde{g}) = A, |p_g| \leq s \right\}
\end{split}
\end{equation*}
where both terms in the product are bounded above by $1$. It follows by Lemma \ref{l:pgsmall} that both
\begin{equation*} \label{e:A1} \tag{$\ast$}
\frac{|\{ g \in B(n) \mid \supp(\tilde{g}) = A, |p_g| \leq s \}|}{\bb(n)} \geq \tilde\varepsilon
\end{equation*}
and
\begin{equation*} \label{e:linkA1} \tag{$\dagger$}
\max \left\{ \frac{|C_G(g) \cap B(n)|}{\bb(n)} \:\middle|\: g \in B(n), \supp(\tilde{g}) = A, |p_g| \leq s \right\} \geq \tilde\varepsilon
\end{equation*}
hold for infinitely many $n$.

The aim is now to show that \eqref{e:A1} and \eqref{e:linkA1} imply that the special subgroups $G_A$ and $G_{\link A}$ (respectively) are non-negligible in $G$. For the latter, one may consider explicit forms of centralisers of $G$: see the next subsection. For the former, note that the set in the numerator consists of elements $g \in B_G(n)$ which have an expression $g = p_g^{-1}\tilde{g}p_g$ with $p_g \in B_G(s)$ and $\tilde{g} \in B_{G_A}(n)$. It follows that
\begin{equation*}
|\{ g \in B(n) \mid \supp(\tilde{g}) = A, |p_g| \leq s \}| \leq \bb_G(s) \bb_{G_A}(n)
\end{equation*}
and so \eqref{e:A1} implies that
\begin{equation*} \label{e:A2} \tag{$\ast\ast$}
\frac{\tilde\varepsilon}{\bb_G(s)} \leq \frac{\bb_{G_A}(n)}{\bb_G(n)} \leq 1
\end{equation*}
for infinitely many $n$, where the second inequality comes from the fact that $B_{G_A}(n) \subseteq B_G(n)$.

\subsection{\texorpdfstring{Centralisers in $G$}{Centralisers in G}}
\label{ss:centr}

In order to use \eqref{e:linkA1}, one needs to consider forms of centralisers of elements $g \in G$ with $\supp(\tilde{g}) = A$. Fix an element $g \in G$ with $\supp(\tilde{g}) = A$ and note that one clearly has $C_G(g) = p_g^{-1}C_G(\tilde{g})p_g$, so if $|p_g| \leq s$ then one has
\begin{equation}
\label{e:centconj}
|C_G(g) \cap B(n)| \leq |C_G(\tilde{g}) \cap B(n+2s)|.
\end{equation}
In particular, it follows from \eqref{e:linkA1} that for infinitely many $n$, there exists an element $g \in B(n)$ with $\supp(\widetilde{g}) = A$ and $|p_g| \leq s$ such that
\begin{equation*} \label{e:linkA2} \tag{$\dagger\dagger$}
\tilde\varepsilon \leq \frac{|C_G(\tilde{g}) \cap B_G(n+2s)|}{\bb_G(n)} \leq \bb_G(2s);
\end{equation*}
here the second inequality comes from the fact that $|C_G(\tilde{g}) \cap B(n+2s)| \leq \bb(n+2s) \leq \bb(n)\bb(2s)$.

Now define an element $g \in G$ to be \emph{cyclically normal} (in the sense of \cite{barkauskas}) if either $\ell_n(g) \leq 1$, or $n := \ell_n(g) \geq 2$ and for any normal form $w = w_1 \cdots w_n \in X^\ast$ of $g$, where $w_i \in X(v_i)^\ast$ for some $v_i \in V(\Gamma)$, one has $v_1 \neq v_n$. Then one has

\begin{lem} \label{l:cn}
For any $g \in G$ with $\supp(\tilde{g}) = A$, there exists an element $\tilde{p}_g \in G_A$ such that $\hat{g} := \tilde{p}_g \tilde{g} \tilde{p}_g^{-1}$ is cyclically normal and $\supp(\hat{g}) = A$.
\end{lem}
\begin{proof}
If $\ell_n(\tilde{g}) \leq 1$ then $\tilde{p}_g = 1$ does the job. Thus suppose that $n := \ell_n(\tilde{g}) \geq 2$. Let
\begin{equation*}
\begin{aligned}
E(\tilde{g}) := \{ g_0 \mid ~&w = w_1 \cdots w_n \in X^\ast \text{ is a normal form for } \tilde{g} \text{ where } \\ &w_i \in X(v_i)^\ast \text{ for some } v_i \in V(\Gamma), \text{ and } w_n \text{ represents } g_0 \}
\end{aligned}
\end{equation*}
be a finite subset of $G_A$. By Proposition \ref{p:wbeh}, any two elements in $E(\tilde{g})$ commute, and so, for any two distinct elements $g_1 \in \hh(v_1)$ and $g_2 \in \hh(v_2)$ of $E(\tilde{g})$, one has $v_1 \neq v_2$. Now define $\tilde{p}_g := \prod_{g_n \in E(\tilde{g})} g_n$. Then $\tilde{p}_g \in G_A$, and following the proof of \cite[Lemma 23]{barkauskas} one can see that $\hat{g} := \tilde{p}_g \tilde{g} \tilde{p}_g^{-1}$ is cyclically normal. Since $\supp(\tilde{p}_g) \subseteq A$ and $\supp(\tilde{g}) = A$, it is clear that $\supp(\hat{g}) \subseteq A$. It also follows by \cite[Lemma 18]{barkauskas} that $\supp(\hat{g}) \cup \supp(\tilde{p}_g) \supseteq A$. Thus one only needs to check that $\supp(\tilde{p}_g) \subseteq \supp(\hat{g})$.

Suppose for contradiction that there exists some $v \in \supp(\tilde{p}_g) \setminus \supp(\hat{g})$, and let $g_v \in E(\tilde{g}) \cap \hh(v)$ be the (unique) element. It is easy to see that $v \notin \link(A \setminus \{v\})$: otherwise any normal form of $\hat{g}$ would contain a subword in $X(v)^\ast$ representing $g_v$ and so $v \in \supp(\hat{g})$. Then, following again the proof of \cite[Lemma 23]{barkauskas}, one has $\tilde{n} := \ell_n(g_v \tilde{g} g_v^{-1}) \leq n-1$, with $\tilde{n} = n-1$ if and only if $\tilde{g}$ has no normal form $w_1 \cdots w_n$, where $w_i \in X(v_i)^\ast$ for some $v_i \in V(\Gamma)$, with $w_1$ and $w_n$ representing $g_v^{-1}$ and $g_v$, respectively. Thus, by minimality of $|\tilde{g}|$, clearly $\tilde{n} = n-1$; but this cannot happen by \cite[Lemma 18]{barkauskas}, since by assumption $v \notin \supp(\hat{g})$. Hence $\supp(\tilde{p}_g) \subseteq \supp(\hat{g})$, as required.
\end{proof}

The following Proposition describes growth of centralisers in $G$.

\begin{prop} \label{p:cent}
Let $g,\tilde{g} \in G$ and $A \subseteq V(\Gamma)$ be as above. Then 
\[
C_G(\tilde{g}) = H_1 \times \cdots \times H_k \times G_{\link A}
\]
for some subgroups $H_1,\ldots,H_k \leq G$, and the following hold:
\begin{enumerate}[(i)]
%\item \label{i:lcentdj} for $1 \leq i \leq k$, the subsets $A_i := \bigcup_{h_i \in H_i} \supp(h_i) \subseteq V(\Gamma)$ are pairwise disjoint and disjoint from $\link A$;
\item \label{i:lcentl} for any $h_1 \in H_1, \ldots, h_k \in H_k$ and $c \in G_{\link A}$,
\[
|h_1 \cdots h_k c|_X = |h_1|_X + \cdots + |h_k|_X + |c|_X;
\]
\item \label{i:lcentb} there exist constants $D_1,\ldots,D_k,\alpha_1,\ldots,\alpha_k \in \Z_{\geq 1}$ such that
\[
|H_i \cap B_{G,X}(n)| \leq D_i n^{\alpha_i}
\]
for all $n \geq 1$.
\end{enumerate}
Furthermore, the number $k \in \Z_{\geq 1}$, the $D_i$ and the $\alpha_i$ only depend on $A$ and not on $g$.
\end{prop}

\begin{proof}
Let $A_1,\ldots,A_k \subseteq A$ form a partition of $A$ such that the graphs $\Gamma(A_i)^C$ are precisely the connected components of the graph $\Gamma(A)^C$, where $\Delta^C$ denotes the complement of a graph $\Delta$. Let $\tilde{p}_g,\hat{g} \in G_A$ be as in Lemma \ref{l:cn}. Then $\supp(\hat{g}) = A$ and so $\hat{g}$ can be expressed as
\begin{equation*}
\hat{g} = \hat{g}_1 \cdots \hat{g}_k
\end{equation*}
where $\supp(\hat{g}_i) = A_i$.

Now suppose without loss of generality that for some $m$, the sets $A_i = \{ v_i \}$ are singletons for $1 \leq i \leq m$, and $|A_i| \geq 2$ for $m+1 \leq i \leq k$. Then Proposition 25, Theorem 32 and Theorem 52 in \cite{barkauskas} state that the centraliser of $\hat{g}$ in $G$ is
\begin{equation*}
C_G(\hat{g}) = C_{\hh(v_1)}(\hat{g}_1) \times \cdots \times C_{\hh(v_m)}(\hat{g}_m) \times \langle h_{m+1} \rangle \times \cdots \times \langle h_k \rangle \times G_{\link A}
\end{equation*}
where $h_{m+1},\ldots,h_k \in G$ are some infinite order elements with $\supp(h_i) = A_i$ (in fact, one has $\hat{g}_i = h_i^{\beta_i}$ for some $\beta_i \in \Z \setminus \{0\}$).

In particular, since $\tilde{p}_g \in G_A$, one has $\tilde{p}_g = p_1 \cdots p_k$ for some $p_i \in G_{A_i}$. Thus $\tilde{p}_g^{-1} q_i \tilde{p}_g = p_i^{-1} q_i p_i$ for any $q_i \in G_{A_i}$, and $\tilde{p}_g^{-1} (G_{\link A}) \tilde{p}_g = G_{\link A}$, hence
\begin{equation*}
%\label{e:centr2}
\begin{aligned}
C_G(\tilde{g}) = \tilde{p}_g^{-1} C_G(\hat{g}) \tilde{p}_g = ~&C_{\hh(v_1)}(\tilde{g}_1) \times \cdots \times C_{\hh(v_m)}(\tilde{g}_m) \\ &\times \langle \tilde{g}_{m+1} \rangle \times \cdots \times \langle \tilde{g}_k \rangle \times G_{\link A}
\end{aligned}
\end{equation*}
where $\tilde{g}_i := p_i^{-1} \hat{g}_i p_i$ for $1 \leq i \leq m$, and $\tilde{g}_i := p_i^{-1} h_i p_i$ for $m+1 \leq i \leq k$. Hence, by setting $H_i := C_{\hh(v_i)}(\tilde{g}_i)$ for $1 \leq i \leq m$ and $H_i := \langle \tilde{g}_i \rangle \cong \Z$ for $m+1 \leq i \leq k$ one obtains the required expression.
%Part \eqref{i:lcentdj} follows from construction, as does the fact that
By construction,
$k$ depends only on $A$ (and not on $g$).

To show \eqref{i:lcentl}, it is enough to note that $H_i \leq G_{A_i}$ for each $i$, and that by construction the subsets $A_i$ are pairwise disjoint and disjoint from $\link A$. Indeed, then it follows from Proposition \ref{p:wbeh} that if $w_i$ (respectively $u$) is a normal form for an element $h_i \in G_{A_i}$ (respectively $c \in G_{\link A}$), then $w_1 \cdots w_k u$ is a normal form for the element $h_1 \cdots h_k c$. This implies \eqref{i:lcentl}.

To show \eqref{i:lcentb} and the last part of the Proposition, one may consider cases $1 \leq i \leq m$ and $m+1 \leq i \leq k$ separately. For $1 \leq i \leq m$, note that, as a consequence of Proposition \ref{p:wbeh}, $|h|_X = |h|_{X(v_i)}$ for all $h \in H_i$, and therefore $|H_i \cap B_{G,X}(n)| = |H_i \cap B_{\hh(v_i),X(v_i)}(n)|$ for all $n \geq 1$. Thus, \eqref{i:lcentb} follows from the facts that $\tilde{g}_i \neq 1$ and that $(\hh(v_i),X(v_i))$ is a rational pair with small centralisers; it also follows that $D_i,\alpha_i$ do not depend on $g$. For $m+1 \leq i \leq k$, it follows from the proof of \cite[Lemma 37]{barkauskas} that since $\hat{g}_i$ is cyclically normal and since $\Gamma(\supp(\hat{g}_i))^C = \Gamma(A_i)^C$ is connected, one has $\ell_n(\tilde{g}_i^\gamma) \geq \ell_n(\hat{g}_i^\gamma) = |\gamma| \ell_n(\hat{g}_i)$ for all $\gamma \in \Z$. In particular, $|\tilde{g}_i^\gamma|_X \geq \ell_n(\tilde{g}_i^\gamma) \geq |\gamma|$ for any $\gamma \in \Z$ and so $|H_i \cap B_{G,X}(n)| \leq 2n+1 \leq 3n$ for all $n \geq 1$. Thus taking $D_i = 3$ and $\alpha_i = 1$ shows \eqref{i:lcentb}; independence from $g$ is clear.
\end{proof}

\subsection{Products of special subgroups}
\label{ss:prod}

To finalise the proof, one employs the following general result:
\begin{lem} \label{l:gr}
Let $G$ be a group with a finite generating set $X$. Let $H,K \leq G$ be subgroups such that $H \times K$ is also a subgroup of $G$, i.e.\ the map $H \times K \to G, (h,k) \mapsto hk$ is an injective group homomorphism. Suppose that there exist constants $\alpha_H,\alpha_K \in \Z_{\geq 0}$, $\lambda_H,\lambda_K \in [1,\infty)$ and $D > C \geq 0$ such that
\begin{align*}
C n^{\alpha_H} \lambda_H^n &\leq |H \cap S_{G,X}(n)| \leq D n^{\alpha_H} \lambda_H^n \\
\text{and} \qquad C n^{\alpha_K} \lambda_K^n &\leq |K \cap S_{G,X}(n)| \leq D n^{\alpha_K} \lambda_K^n
\end{align*}
for all $n \geq 1$. Furthermore, suppose that $|hk|_X = |h|_X + |k|_X$ for all $h \in H^{(n)}$, $k \in K^{(n)}$, and that $\lambda_H \geq \lambda_K$. If $\lambda_H > \lambda_K$, then there exists constant $\widetilde{D} = \widetilde{D}(D,\alpha_H,\alpha_K,\lambda_H,\lambda_K) > 0$, which does not depend on $H$ or $K$, such that
\[
|(H \times K) \cap S_{G,X}(n)| \leq \widetilde{D} n^{\alpha_H} \lambda_H^n
\]
for all $n \geq 1$. Furthermore, if $\lambda_H = \lambda_K$ and $C > 0$, then no such constant $\widetilde{D}$ exists.
\end{lem}

% The `if' direction of the Lemma is true even for $C = 0$.

\begin{proof}
Suppose first that $\lambda_H > \lambda_K$. Clearly it is enough to show that
\begin{equation*}
\limsup_{n \to \infty} \frac{|(H \times K) \cap S_{G,X}(n)|}{n^{\alpha_H}\lambda_H^n} < \infty.
\end{equation*}
Fix $n \geq 1$. As $|hk|_X = |h|_X + |k|_X$ for any $h \in H$, $k \in K$, one has
\begin{equation*}
\begin{split}
&\frac{|(H \times K) \cap S_{G,X}(n)|}{n^{\alpha_H}\lambda_H^n} = \frac{1}{n^{\alpha_H}\lambda_H^n} \sum_{i=0}^n |H \cap S_{G,X}(n-i)| \times |K \cap S_{G,X}(i)| \\ &\qquad \leq D^2 \left( 1 + \sum_{i=1}^{n-1} \left( \frac{\lambda_K}{\lambda_H} \right)^i \left(\frac{n-i}{n}\right)^{\alpha_H} i^{\alpha_K} + \left( \frac{\lambda_K}{\lambda_H} \right)^n n^{\alpha_K-\alpha_H} \right).
\end{split}
\end{equation*}
As $\lambda_K/\lambda_H < 1$, limits of the first and third term above as $n \to \infty$ are $D^2$ and $0$, respectively. The second term can be bounded above by an upper bound for the series $D^2 \sum_i (\lambda_K/\lambda_H)^i i^{\alpha_K}$, which converges by the ratio test. Hence indeed $\limsup_{n \to \infty} |(H \times K) \cap S_{G,X}(n)|/(n^{\alpha_H}\lambda_H^n) < \infty$, which implies the result. It is also clear from the inequality above that $\widetilde{D}$ depends only on $D$, $\alpha_H$, $\alpha_K$, $\lambda_H$ and $\lambda_K$.

Conversely, suppose that $C > 0$ and $\lambda_H = \lambda_K =: \lambda$. Let $n \geq 20$, so that $\lceil \sqrt{n} \rceil \leq n/4$. Then
\begin{equation*}
\begin{split}
\frac{|(H \times K) \cap S_{G,X}(n)|}{n^{\alpha_H}\lambda^n} &= \frac{1}{n^{\alpha_H}\lambda^n} \sum_{i=0}^n |H \cap S_{G,X}(n-i)| \times |K \cap S_{G,X}(i)| \\ &\geq C^2 \sum_{i=1}^{n-1} \left(\frac{n-i}{n}\right)^{\alpha_H} i^{\alpha_K} \geq C^2 \sum_{i=\lceil \sqrt{n} \rceil}^{\lfloor n/2 \rfloor} \left(\frac{1}{2}\right)^{\alpha_H} (\sqrt{n})^{\alpha_K} \\ &\geq C^2 2^{-(\alpha_H+2)} n^{\frac{\alpha_K}{2} + 1}.
\end{split}
\end{equation*}
In particular, one has $|(H \times K) \cap S_{G,X}(n)|/(n^{\alpha_H}\lambda_H^n) \to \infty$ as $n \to \infty$, implying the result.
\end{proof}

Given this Lemma, the proof can be finalised as follows. Recall (see \eqref{e:Sbound} and \eqref{e:Bbound}) that one has constants $\alpha \in \Z_{\geq 0}$, $\lambda > 1$ and $D_{V(\Gamma)} > C_{V(\Gamma)} > 0$ such that
\begin{equation}
\label{e:Ggr}
\begin{split}
&C_{V(\Gamma)} n^\alpha \lambda^n \leq \sss_{G}(n) \leq D_{V(\Gamma)} n^\alpha \lambda^n \\
\text{and} \quad &C_{V(\Gamma)} n^\alpha \lambda^n < \bb_{G}(n) < \frac{D_{V(\Gamma)}\lambda}{\lambda-1} n^\alpha \lambda^n
\end{split}
\end{equation}
for all $n \geq 1$. Now \eqref{e:A2} implies that, for infinitely many $n$,
\begin{equation} \label{e:A30}
\widetilde{C}_A n^\alpha \lambda^n \leq \bb_{G_A}(n) \leq \widetilde{D}_A n^\alpha \lambda^n
\end{equation}
for some $\widetilde{D}_A > \widetilde{C}_A > 0$. But as $G_A$ has rational growth with respect to $\bigsqcup_{v \in A} X(v)$, it follows from Theorem \ref{t:growth} that in fact, after modifying the constants $\widetilde{D}_A$ and $\widetilde{C}_A$ if necessary, \eqref{e:A30} holds for all $n \geq 1$, and since $\lambda > 1$, after further modifying $\widetilde{C}_A$, one has
\begin{equation*} \label{e:A3} \tag{$\ast\!\ast\!\ast$}
\widetilde{C}_A n^\alpha \lambda^n \leq \sss_{G_A}(n) \leq \widetilde{D}_A n^\alpha \lambda^n
\end{equation*}
for all $n \geq 1$.

Moreover, \eqref{e:linkA2} implies that for infinitely many $n \geq 2s+1$ there exists $g \in B(n)$ such that
\begin{equation*}
%\label{e:linkAl0a0}
\widetilde{C}' (n-2s)^\alpha \lambda^{n-2s} \leq |C_G(\tilde{g}) \cap B_{G,X}(n)| \leq \widetilde{D}' (n-2s)^\alpha \lambda^{n-2s}
\end{equation*}
for some $\widetilde{D}' > \widetilde{C}' > 0$. After decreasing the constant $\widetilde{C}' > 0$ if necessary, one may therefore assume that, for infinitely many $n$,
\begin{equation}
\label{e:linkAC}
\widetilde{C}' n^\alpha \lambda^n \leq |C_G(\tilde{g}) \cap B_{G,X}(n)| \leq \widetilde{D}' n^\alpha \lambda^n
\end{equation}
for some $g \in B(n)$ with $\supp(\widetilde{g}) = A$ and $|p_g| \leq s$.

Note that $G_{\link A}$ has rational growth with respect to $\bigsqcup_{v \in \link A} X(v)$ as it is a special subgroup of $G$, and so by Theorem \ref{t:growth} it follows that, for all $n \geq 1$,
\begin{equation}
\label{e:linkAl0a0}
\widetilde{C}_{\link A} n^{\alpha_0} \lambda_0^n \leq \sss_{G_{\link A}}(n) \leq \widetilde{D}_{\link A} n^{\alpha_0} \lambda_0^n
\end{equation}
for some $\widetilde{D}_{\link A} > \widetilde{C}_{\link A} > 0$ and some $\alpha_0 \in \Z_{\geq 0}$, $\lambda_0 \geq 1$.

One may now show that $(\lambda_0,\alpha_0) = (\lambda,\alpha)$. Indeed, as $\sss_{G_{\link A}}(n) \subseteq \sss_G(n)$, it follows from \eqref{e:Ggr} that either $\lambda_0 < \lambda$ or $\lambda_0 = \lambda$ and $\alpha_0 \leq \alpha$. Let $g \in G$ be such that $\supp(\widetilde{g}) = A$ for all $n$. By Proposition \ref{p:cent}, one has an expression
\[
C_G(\widetilde{g}) = H_1 \times \cdots \times H_k \times G_{\link A}.
\]

One now applies Lemma \ref{l:gr} $k$ times. In particular, for each $i=k,k-1,\ldots,1$ in order, it follows from Proposition \ref{p:cent} that Lemma \ref{l:gr} can be applied for
\begin{align*}
H &:= H_{i+1} \times \cdots \times H_k \times G_{\link A}, \\
K &:= H_i \\
(\alpha_H,\lambda_H) &:= \begin{cases} (\alpha_0,\lambda_0) & \text{if } \lambda_0 > 1, \\ (0,\frac{\lambda+1}{2}) & \text{if } \lambda_0 = 1, \end{cases} \\
(\alpha_K,\lambda_K) &:= \left(0,\frac{\lambda_H+1}{2}\right), \\
C &:= 0, \\
\text{and} \qquad D = \overline{D}_i &:= \max\{\widetilde{D}_i,\widetilde{D}_i'\}.
\end{align*}
Here $\widetilde{D}_i'>0$ is such that $D_i n^{\alpha_i} \leq \widetilde{D}_i' \lambda_K^n$ for each $n \geq 1$, where $D_i$ and $\alpha_i$ are as in Proposition \ref{p:cent}, $\widetilde{D}_k$ is such that $\sss_{G_{\link A}}(n) \leq \widetilde{D}_k n^{\alpha_H} \lambda_H^n$ for all $n \geq 1$, and, for each $i=k-1,k-2,\ldots,1$, $\widetilde{D}_i = \widetilde{D}(\overline{D}_{i+1},\alpha_H,\alpha_K,\lambda_H,\lambda_K)$ is the constant given by Lemma \ref{l:gr}.

It then follows that, for all $g \in G$ with $\supp(\widetilde{g}) = A$ and $|p_g| \leq s$,
\begin{equation}
\label{e:CGgsmall}
|C_G(\widetilde{g}) \cap S_{G,X}(n)| \leq \widetilde{D} n^{\alpha_H} \lambda_H^n
\end{equation}
for all $n \geq 1$, where $\widetilde{D} = \widetilde{D}(\overline{D}_1,\alpha_H,\alpha_K,\lambda_H,\lambda_K)$ is the constant, independent from $g$, given by Lemma \ref{l:gr}. Since $\lambda_H > 1$, by further increasing $\widetilde{D}$ we may replace $S_{G,X}(n)$ with $B_{G,X}(n)$ in \eqref{e:CGgsmall}. But by construction, one has either $\lambda_H < \lambda$ or $\lambda_H = \lambda$ and $\alpha_H \leq \alpha$, and so together with \eqref{e:linkAC} this implies that $(\lambda_H,\alpha_H) = (\lambda,\alpha)$. Thus, by the choice of $(\lambda_H,\alpha_H)$, one has $(\lambda_0,\alpha_0) = (\lambda,\alpha)$, as claimed. In particular, \eqref{e:linkAl0a0} can be rewritten as
\begin{equation*} \label{e:linkA3} \tag{$\dagger\!\dagger\!\dagger$}
\widetilde{C}_{\link A} n^\alpha \lambda^n \leq \sss_{G_{\link A}}(n) \leq \widetilde{D}_{\link A} n^\alpha \lambda^n.
\end{equation*}

Finally, note that the group $G_{A \cup \link A} = G_A \times G_{\link A}$ is a special subgroup of $G$ and so one has $S_{G_{A \cup \link A}}(n) \subseteq S_G(n)$. It then follows from \eqref{e:A3}, \eqref{e:linkA3} and Lemma \ref{l:gr} that for any $\widetilde{D} > 0$ one has
\begin{equation*}
\sss_G(n) \geq \sss_{G_{A \cup \link A}}(n) > \widetilde{D} n^\alpha \lambda^n
\end{equation*}
for some $n$, which contradicts \eqref{e:Ggr}. This completes the proof of Theorem \ref{t:dc}.

\bibliographystyle{amsplain}
\bibliography{../../../../all}

\providecommand{\bysame}{\leavevmode\hbox to3em{\hrulefill}\thinspace}
\providecommand{\MR}{\relax\ifhmode\unskip\space\fi MR }
\providecommand{\MRhref}[2]{%
  \href{http://www.ams.org/mathscinet-getitem?mr=#1}{#2}
}
\providecommand{\href}[2]{#2}
\begin{thebibliography}{10}

\bibitem{amv}
Y.~Antol\'{i}n, A.~Martino, and E.~Ventura, \emph{Degree of commutativity of
  infinite groups}, Proc. Amer. Math. Soc. \textbf{145} (2017), 479--485.

\bibitem{barkauskas}
D.~A. Barkauskas, \emph{Centralisers in graph products of groups}, J. Algebra
  \textbf{312} (2007), no.~1, 9--32.

\bibitem{benson}
M.~Benson, \emph{Growth series of finite extensions of {}$\mathbb{Z}^n$ are
  rational}, Invent. Math. \textbf{73} (1983), 251--269.

\bibitem{berstel}
J.~Berstel, \emph{Sur les p\^oles et le quotient de {H}adamard de s\'eries
  $\mathbf{N}$-rationnelles}, C. R. Acad. Sci. A \textbf{272} (1971),
  1079--1081.

\bibitem{besicovitch}
A.~S. Besicovitch, \emph{Almost periodic functions}, Dover publications, 1954.

\bibitem{cannon}
J.~W. Cannon, \emph{The combinatorial structure of cocompact discrete
  hyperbolic groups}, Geom. Dedicata \textbf{16} (1984), no.~2, 123--148.

\bibitem{chiswell}
I.~M. Chiswell, \emph{The growth series of a graph product}, B. Lond. Math.
  Soc. \textbf{26} (1994), no.~3, 268--272.

\bibitem{coornaert}
M.~Coornaert, \emph{Mesures de {P}atterson-{S}ullivan sur le bord d'un espace
  hyperbolique au sens de {G}romov}, Pacific J. Math. \textbf{159} (1993),
  no.~2, 241--270.

\bibitem{duchin}
M.~Duchin and M.~Shapiro, \emph{Rational growth in the {H}eisenberg group},
  preprint, available at
  \href{https://arxiv.org/abs/1411.4201}{arXiv:1411.4201} [math.GR], 2014.

\bibitem{erdos}
P.~Erd\H{o}s and P.~Tur\'an, \emph{On some problems of a statistical
  group-theory, {IV}}, Acta Math. Acad. Sci. H. \textbf{19} (1968), no.~3--4,
  413--435.

\bibitem{green}
E.~R. Green, \emph{Graph products of groups}, Ph.D. thesis, The University of
  Leeds, 1990.

\bibitem{griha}
R.~Grigorchuk and P.~de~la Harpe, \emph{On problems related to growth, entropy,
  and spectrum in group theory}, J. Dyn. Control Syst. \textbf{3} (1997),
  no.~1, 51--89.

\bibitem{gromov81}
M.~Gromov, \emph{Groups of polynomial growth and expanding maps}, Publ. Math.
  I. H. \'E. S. \textbf{53} (1981), 53--78.

\bibitem{gromov87}
\bysame, \emph{Hyperbolic groups}, Essays in Group Theory (S.M. Gersten, ed.),
  Mathematical Sciences Research Institute Publications, vol.~8,
  Springer-Verlag, 1987, pp.~75--263.

\bibitem{gustafson}
W.~H. Gustafson, \emph{What is the probability that two group elements
  commute?}, Amer. Math. Monthly \textbf{80} (1973), no.~9, 1031--1034.

\bibitem{lee}
D.~V. Lee, \emph{On the power-series expansion of a rational function}, Acta
  Arith. \textbf{LXII} (1992), no.~3, 229--255.

\bibitem{pansu}
P.~Pansu, \emph{Croissance des boules et des g\'eod\'esiques ferm\'ees dans les
  nilvari\'et\'es}, Ergodic Theory Dynam. Systems \textbf{3} (1983), 415--445.

\bibitem{stoll95}
M.~Stoll, \emph{Regular geodesic languages for 2-step nilpotent groups},
  Combinatorial and Geometric Group Theory, Edinburgh 1993 (A.~J. Duncan, N.~D.
  Gilbert, and J.~Howie, eds.), London Mathematical Society Lecture Note
  Series, vol. 204, Cambridge University Press, 1995, pp.~294--299.

\bibitem{stoll96}
\bysame, \emph{Rational and transcendental growth series for the higher
  {H}eisenberg groups}, Invent. Math. \textbf{126} (1996), 85--109.

\bibitem{wolf}
J.~A. Wolf, \emph{Growth of finitely generated solvable groups and curvature of
  {R}iemannian manifolds}, J. Diff. Geom. \textbf{2} (1968), 421--446.

\end{thebibliography}

\end{document}